\pdfoutput=1
\documentclass[10pt]{amsart}

\numberwithin{equation}{section}

\usepackage{amsmath,amsfonts,amssymb}
\usepackage{mathtools}
\usepackage[utf8]{inputenc}
\usepackage[english]{babel}
\usepackage{hyperref}
\usepackage{tikz-cd}

\DeclareUnicodeCharacter{FB01}{fi}
\newtheorem{theorem}{Theorem}[section]

\newtheorem{lemma}[theorem]{Lemma}
\newtheorem{prop}{Proposition}

\theoremstyle{definition}
\newtheorem{definition}{Definition}

\begin{document}
	\title[TORELLI THEOREM FOR SYMPLECTIC PARABOLIC HIGGS BUNDLES]
	{TORELLI THEOREM FOR THE MODULI SPACE OF SYMPLECTIC PARABOLIC HIGGS BUNDLES}
	
	\author{Sumit Roy}
	\address{School of Mathematics, Tata Institute of Fundamental Research, Homi Bhabha Road, Colaba, Mumbai 400005, India.}
	\email{sumit@math.tifr.res.in}
	\thanks{E-mail : sumit@math.tifr.res.in}
	\thanks{Address : School of Mathematics, Tata Institute of Fundamental Research, Homi Bhabha Road, Colaba, Mumbai 400005, India.}
	
	\begin{abstract}
	 Let $(X,D)$ and $(X',D')$ be two  compact Riemann surfaces of genus $g \geq 4$ with the set of marked points $D \subset X$ and $D' \subset X'$. Fix a parabolic line bundle $L$ with trivial parabolic structure. Let $\mathcal{N}_{\textnormal{Sp}}(2m,\alpha,L)$ and $\mathcal{N}'_{\textnormal{Sp}}(2m,\alpha,L)$ be the moduli spaces of stable symplectic parabolic Higgs bundles over $X$ and $X'$ respectively, with rank $2m$ and fixed parabolic structure $\alpha$, with the symplectic form taking values in $L$. We prove that if $\mathcal{N}_{\textnormal{Sp}}(2m,\alpha,L)$ is isomorphic to $\mathcal{N}'_{\textnormal{Sp}}(2m,\alpha,L)$, then there exist an isomorphism between $X$ and $X'$ sending $D$ to $D'$. 	
	\end{abstract}
	\maketitle	
	\section{Introduction}
	Let $X$ and $X'$ be two compact connected Riemann surfaces of genus $g$. The classical Torelli theorem says that if their Jacobians $J(X)$ and $J(X')$ are isomorphic as polarized varieties, with the polarization given by the theta divisor, then $X$ is isomorphic to $X'$.\\
	In \cite{MN68} Mumford and Newstead proved a similar result for $\mathcal{M}_X^{2,\xi}$, the moduli space of stable vector bundles over $X$  with rank $2$ and fixed determinant $\xi$ (assuming $g\geq 2$). Later, in \cite{NR75} Narasimhan and Ramanan extended this result for any rank. They consider the intermediate Jacobian associated to $H^3(\mathcal{M}_X^{r,\xi})$. They showed that this has a polarization defined by the positive generator of $\textnormal{Pic}(\mathcal{M}_X^{r,\xi})$, and this polarized intermediate Jacobian is isomorphic (as a polarized variety) to the Jacobian of the curve $X$, with the polarization given by the theta divisor. So the result follows from the classical Torelli theorem. A Torelli theorem for the moduli space of symplectic bundles was proved in \cite{BH12} and \cite{BGM12}.\\
	The notion of parabolic bundles over a curve was described by Mehta and Seshadri in \cite{MS80}, where they also constructed its moduli space using Geometric Invariant Theory. In \cite{BR89} Bhosle and Ramanathan extended this notion to parabolic $G$-bundles where $G$ is a connected reductive group, and constructed their moduli space. The notion of symplectic parabolic bundles were described in \cite{BMW11}. They are parabolic vector bundles with a nondegenerate (in a suitable sense) anti-symmetric form taking values in a parabolic line bundle.\\
	Hitchin, in \cite{H87} showed that the moduli space of stable Higgs bundles over a curve $X$ forms an algebraically completely integrable system fibered, over a space of invariant polynomials, either by a Jacobian or a Prym variety of spectral curves. Later, in \cite{M94} Markman extended this result for the moduli space of stable $L$-twisted Higgs bundles where $L$ is a positive line bundle on $X$ satisfying $L \geq K_X$.\\
	In \cite{BBB01} Balaji, Baño and Biswas proved a Torelli theorem for parabolic bundles of rank $2$, and in \cite{BG03} Biswas and Gómez proved a Torelli theorem for Higgs bundles (with genus $g\geq 2$). Gómez and Logares in their paper \cite{GL11} proved a Torelli theorem for parabolic Higgs bundles of rank $2$ by applying the Torelli thoerem of \cite{BBB01}. In \cite{AG19}, Alfaya and Gómez proved a Torelli theorem for parabolic bundles of any rank (assuming $g \geq 4$). Our goal in this article is to prove a Torelli theorem for symplectic parabolic Higgs bundles.\\
	The main result is the following theorem (see section $2$ for notation):
	\begin{theorem}\label{thm2}
	 Let $(X,D)$ and $(X',D')$ be two  compact Riemann surfaces of genus $g \geq 4$ with the set of marked points $D \subset X$ and $D' \subset X'$.  Let $\mathcal{N}_{\textnormal{Sp}}(2m,\alpha,L)$ and $\mathcal{N}'_{\textnormal{Sp}}(2m,\alpha,L)$ be the moduli spaces of stable symplectic parabolic Higgs bundles over $X$ and $X'$, respectively. If $\mathcal{N}_{\textnormal{Sp}}(2m,\alpha,L)$ is isomorphic to $\mathcal{N}'_{\textnormal{Sp}}(2m,\alpha,L)$, then $(X,D)$ is isomorphic to $(X',D')$, i.e. there exist an isomorphism $X \cong X'$ sending $D$ to $D'$.  
	\end{theorem}

	In section $2$ we recall the notion of symplectic parabolic bundles and some properties of their moduli space. We also prove some technical lemmas regarding $(1,0)$-stable symplectic parabolic bundles. In section $3$ we give a description of the \textit{Hitchin map} and then we study the locus of the singular spectral curves, called the \textit{Hitchin discriminant}. We can intrinsically describe the image of the Hitchin discriminant as an abstract variety from the geometry of $\mathcal{M}_{\textnormal{Sp}}(2m,\alpha,L)$, the moduli space of stable symplectic parabolic bundles. In section $4$ we will use this description to prove the Torelli theorem for the moduli space of symplectic parabolic bundles (Theorem \ref{thm1}).\\
	In section $5$ we again consider the Hitchin map for $\mathcal{N}_{\textnormal{Sp}}(2m,\alpha,L)$, the moduli space of symplectic parabolic Higgs bundles. The fiber over the origin is called the nilpotent cone. The moduli space $\mathcal{M}_{\textnormal{Sp}}(2m,\alpha,L)$ is embedded in the nilpotent cone as a component and we will see that it is the unique irreducible component of the nilpotent cone that doesn't admit a non-trivial $\mathbb{C}^*$-action. In other words, if we were given $\mathcal{N}_{\textnormal{Sp}}(2m,\alpha,L)$ and the Hitchin map, we would recover $\mathcal{M}_{\textnormal{Sp}}(2m,\alpha,L)$, and then we can apply the Torelli theorem for $\mathcal{M}_{\textnormal{Sp}}(2m,\alpha,L)$ to recover $X$.\\
	We are given only the isomorphism class of $\mathcal{N}_{\textnormal{Sp}}(2m,\alpha,L)$. So we have to recover the Hitchin map. The idea is same as in \cite{BG03}. Consider an algebraic variety $Y$, which is isomorphic to $\mathcal{N}_{\textnormal{Sp}}(2m,\alpha,L)$. Then the natural map
	\[
	m : Y \to \textnormal{Spec}(\Gamma(Y))
	\]
	is isomorphic to the Hitchin map up to an automorphism. More precisely, by Proposition \ref{prop5}, $\textnormal{Spec}(\Gamma(Y)) \cong \mathbb{A}^N$, where $N= \dim \mathcal{M}_{\textnormal{Sp}}(2m,\alpha,L)$, and there is a commutative diagram \ref{diag1}.\\
	The standard $\mathbb{C}^*$-action on $\mathcal{N}_{\textnormal{Sp}}(2m,\alpha,L)$ given by sending a Higgs pair $(E,\Phi)$ to $(E,t\Phi)$ descends to an action on the Hitchin space, and the origin is the only fixed of this descended action. In section $6$ we use the Kodaira-Spencer map to prove that if $g$ is a $\mathbb{C}^*$-action, having exactly one fixed point, and admitting a lift to $\mathcal{N}_{\textnormal{Sp}}(2m,\alpha,L)$, then this fixed point is the origin of the Hitchin space (see Proposition \ref{prop10}). So using this property we recover the origin of the Hitchin space. Hence we recover the nilpotent cone, and by the previous observations, the proof of the theorem is complete.

	\section{Preliminaries}
	\subsection{Parabolic Vector Bundles}Let $X$ be a compact Riemann surface, and let $D \subset X$ be a finite subset of $n \geq 1$ distinct points. A \textit{parabolic vector bundle} $E_*$ on $X$ is a holomorphic vector bundle $E$ of rank $r$ over $X$ together with a parabolic structure, i.e. for every point $p \in D$, we have
	\begin{enumerate}
		\item a filtration of subspaces  $$E|_p=: E_{p,1}\supsetneq \dots \supsetneq E_{p,r(p)} \supsetneq \{0\}, $$
		\item a sequence of real numbers (parabolic weights)  satisfying $$0\leq \alpha_1(p) < \alpha_2(p) < \dots < \alpha_r(p) < 1.$$
	\end{enumerate}
	The parabolic structure is said to have \textit{full flags} whenever dim$(E_{p,i}/E_{p,i+1}) = 1$ \hspace{1cm}$\forall i, \forall p\in D$. We denote $\alpha = \{(\alpha_1(p),\dots,\alpha_r(p)) \}_{p \in D}$ to the system of weights corresponding to the fixed parabolic structure. \\
	The \textit{parabolic degree} of a parabolic vector bundle $E_*$ is defined as
	\[
	\operatorname{pardeg}(E_*):= \deg(E)+ \sum\limits_{p\in D}\sum\limits_{i} \alpha_i(p) \cdot \dim(E_{p,i}/E_{p,i+1})
	\]
	and the real number $\operatorname{pardeg}(E_*)/\text{rank}(E)$ is called the \textit{parabolic slope} of $E_*$ and it is denoted by $\mu_{par}(E_*)$. The dual of a parabolic bundle and the tensor product of two parabolic bundles can be defined in a natural way (see \cite{Y95}).\\
	A \textit{parabolic homomorphism} $\phi : E_* \to E^\prime_*$ between two parabolic bundles is a homomorphism of vector bundles that satisfies
	the following: at each $p \in D$ we have $\phi_p(E_{p,i}) \subset E_{p,i+1}^\prime$ whenever $\alpha_i(p) > \alpha_{i+1}^\prime(p)$. Furthermore, we call such morphism \textit{strongly parabolic} if $\alpha_i(p) \geq \alpha_{i+1}^\prime(p)$ implies $\phi_p(E_{p,i}) \subset E_{p,i+1}^\prime$ for every $p \in D$.

	\subsection{Symplectic Parabolic Bundles} Fix a parabolic line bundle $L_*$. Let $E_*$ be a parabolic bundle and let
	\[
	\psi : E_* \otimes E_* \to L_*
	\]
	be a homomorphism of parabolic bundles. Tensoring both sides with the parabolic dual $E^\vee_*$ we get a homomorphism
	\[
	\psi \otimes Id : (E_* \otimes E_*) \otimes E^\vee_* \to L_* \otimes E^\vee_*.
	\]
	The trivial line bundle $\mathcal{O}_X$ equipped with the trivial parabolic structure (meaning parabolic weights are all zero) is realized as a parabolic subbundle of $E_* \otimes E^\vee_*$. Let
	\[
	\tilde{\psi} : E_* \to L_* \otimes E^\vee_*
	\]
	be the homomorphism defined by the composition
	\[
	E_* = E_* \otimes \mathcal{O}_X  \xhookrightarrow{} E_* \otimes (E_* \otimes E^\vee_*) = (E_* \otimes E_*) \otimes E^\vee_* \xrightarrow{\psi \otimes Id} L_* \otimes E^\vee_*.
	\]
	\begin{definition} A \textit{symplectic parabolic bundle} is a pair $(E_*,\psi)$ of the above form such that $\psi$ is anti-symmetric and the homomorphism $\tilde{\psi}$ is an isomorphism.
	\end{definition}
	
	\subsection{Parabolic Higgs Bundles} Let $K$ be the canonical bundle on $X$. We write $K(D) \coloneqq K \otimes \mathcal{O}(D)$. A \textit{parabolic Higgs bundle} on $X$ is a parabolic bundle $E_*$ on $X$ together with a Higgs field $\Phi : E_* \to E_* \otimes K(D)$ such that $\Phi$ is strongly parabolic.\
	
	Higgs field associated to a parabolic Higgs bundle is called \textit{parabolic Higgs field}.
	
	\subsection{Symplectic Parabolic Higgs Bundles} Let $(E_*,\psi)$ be a symplectic parabolic bundle on $X$. A parabolic Higgs field on $E_*$ will induce a parabolic Higgs field on $L_* \otimes E^\vee_*$. A parabolic Higgs field $\Phi$ is said to be compatible with  $\psi$ if $\tilde{\psi}$ takes $\Phi$ to the induced parabolic Higgs field on $L_* \otimes E^\vee_*$.\
	
	A \textit{symplectic parabolic Higgs bundle} $(E_*,\Phi,\psi)$ is a symplectic parabolic bundle $(E_*,\psi)$  togther with a parabolic Higgs field $\Phi$ on $E_*$ which is compatible with $\psi$.
	
	\begin{definition} A holomorphic subbundle $F \subset E$ is called \textit{isotropic} if $\psi(F \otimes F) = 0$. The parabolic structure on $E$ will induce a parabolic structure on $F$. A symplectic parabolic Higgs bundle $(E_*,\Phi,\psi)$ is called \textit{stable} (resp. \textit{semistable}) if for every isotropic subbundle $F \subset E$ of positive rank, the following condition holds
		\[
		\mu_{par}(F_*) < \mu_{par}(E_*) \hspace{0.4cm}(\text{resp.} \hspace{0.15cm} \mu_{par}(F_*) \leq \mu_{par}(E_*)).
		\]
	\end{definition}\
	In \cite{BBN01} and \cite{BBN03}, principal bundles with parabolic structures were defined when all parabolic weights are rational. By \cite{BMW11}, the definition of symplectic parabolic bundle coincides with the definition in \cite{BBN01} and \cite{BBN03} when all parabolic weights are rational.\\
	Let 
	\[
	J = \begin{bmatrix} O_{m \times m} & I_{m \times m} \\ -I_{m \times m} & O_{m \times m} \end{bmatrix}
	\]
	be the standard symplectic form on $\mathbb{C}^{2m}$. Consider the group
	\begin{equation}\label{eqn6}
	\text{Gp}(2m,\mathbb{C}) = \{ A \in \text{GL}(2m,\mathbb{C}) : A^tJA=cJ \text{  for some  } c \in \mathbb{C}^*\}.
	\end{equation}
	This group is an extension of $\mathbb{C}^*$ by the symplectic group $\text{Sp}(2m,\mathbb{C})$
	\[
	1 \to \text{Sp}(2m,\mathbb{C}) \xrightarrow{} \text{Gp}(2m,\mathbb{C}) \xrightarrow[]{p} \mathbb{C}^* \xrightarrow[]{} 1,
	\]
	where $p(A)=c$ for $A$ and $c$ as in (\ref{eqn6}). It follows that $\det(A)=p(A)^m$ for all $A \in \text{Gp}(2m,\mathbb{C})$. When all parabolic weights are rational, giving a symplectic parabolic bundle of rank $2m$ is equivalent to giving a parabolic principal $\text{Gp}(2m,\mathbb{C})$-bundle.
	\subsection{Moduli Space of Symplectic Parabolic Higgs bundles}
	The moduli space of stable parabolic $G$-bundles of rank $r$ and degree $d$ and fixed parabolic structure $\alpha$ was described in \cite{BR89} and \cite{BBN01}. It is a normal projective variety. 	Fix a parabolic line bundle $L$ with trivial parabolic structure.  
	Let $\mathcal{M}_{\text{Sp}}(2m,\alpha,L)$ denote the moduli space of stable symplectic parabolic bundles of rank $2m$ $ (m > 1)$ and fixed parabolic structure $\alpha$, with the symplectic form taking values in $L$. When the parabolic structure $\alpha$ have full flags, it is of dimension
	\[
	\dim \mathcal{M}_{\text{Sp}}(2m,\alpha,L) = m(2m+1)(g-1) + m^2n,
	\]
	where $n$ is the number of marked points on $X$.\\
	The moduli space $\mathcal{N}_{\text{Sp}}(2m,\alpha,L)$ of stable symplectic parabolic Higgs bundles of rank $2m$ is a smooth irreducible complex variety. The moduli space $\mathcal{M}_{\text{Sp}}(2m,\alpha,L)$ is embedded in $\mathcal{N}_{\text{Sp}}(2m,\alpha,L)$, by considering the zero Higgs fields. By the parabolic version of Serre duality, $\mathcal{N}_{\text{Sp}}(2m,\alpha,L)$ contains the cotangent bundle $T^*\mathcal{M}_{\text{Sp}}(2m,\alpha,L)$ as an open subset. Therefore,
	\[
	\dim\mathcal{N}_{\text{Sp}}(2m,\alpha,L) = 2\dim \mathcal{M}_{\text{Sp}}(2m,\alpha,L) = 2m(2m+1)(g-1) + 2m^2n.
	\]
From now on we assume the parabolic structure to have full flags and rational parabolic weights. We also assume that the parabolic weights are small enough so that the stability of the symplectic parabolic Higgs bundle is equivalent to the stability of the underlying vector bundle.
	\begin{definition}
		A symplectic parabolic bundle $(E_*,\psi)$ is $(k,l)$-stable (res. $(k,l)$- semistable) if for all non-zero isotropic subbundles $F \subset E$, 
		\[
		\frac{\text{pardeg}(F_*) + k}{\text{rank}(F)} < \frac{\text{pardeg}(E_*) - l}{\text{rank}(E)} \hspace{1cm}(\text{res.  } \leq )
		\]
		holds.

	\end{definition}
Observe that if $k,l \geq 0$, then a $(k,l)$-stable symplectic parabolic bundle is stable in the usual sense.
	
	\begin{prop}\label{prop1}
		For $ g \geq 3$, the locus of $(1,0)$-stable symplectic parabolic bundles is a non-empty Zariski open subset of $\mathcal{M}_{\textnormal{Sp}}(2m,\alpha,L)$.
	\end{prop}
	\begin{proof} The proof is analogous to the proof of \cite[Proposition 2.7]{BG06}. If a stable symplectic parabolic bundle $(E_*,\psi)$ is not $(1,0)$-stable, then there exist an isotropic subbundle $F\subset E$ such that 
		\[
		\frac{\text{pardeg}(F_*) + 1}{\text{rank}(F)} \geq \frac{\text{pardeg}(E_*)}{\text{rank}(E)}.
		\]
		Or equivalently, $F_*$ satisfies the following
		\[
		\frac{\text{pardeg}(F_*) + 1}{\text{rank}(F)} \geq \frac{\text{pardeg}(E_*/F_*)}{\text{rank}(E/F)}.
		\]
		
		Therefore the ranks, degrees and weight-multiplicities for  quotients $E_*/F_*$ of $E_*$ vary over finite sets. Hence, by the properness of the (parabolic) Quot scheme, the complement of $(1,0)$-stable  bundles in $\mathcal{M}_{\text{Sp}}(2m,\alpha,L)$ is a finite union of closed sets. Thus the locus of $(1,0)$-stable bundles is an open subset of $\mathcal{M}_{\text{Sp}}(2m,\alpha,L)$. So it remains to show that the locus of $(1,0)$-stable bundles is non-empty.\\
		Let $(E_*,\psi)$ be a stable symplectic parabolic bundle which is not $(1,0)$-stable and let $P$ be the parabolic principal $\text{Gp}(2m,\mathbb{C})$-bundle corresponding to $(E_*,\psi)$. Let $F\subset E$ be an isotropic subbundle which contradicts the $(1,0)$-stability condition. This gives a reduction of the structure group $P^H \subset P$ to a maximal parabolic subgroup $H \subset \text{Gp}(2m,\mathbb{C})$.\\
		From the openness of the stability condition, it follows that a deformation of $P^H$ as parabolic principal $H$-bundle will give deformations of $P$ which are not $(1,0)$-stable but are still stable. Also, any deformation of $P$ which is not $(1,0)$-stable comes from a deformation of $P^H$, for some maximal parabolic subgroup $H \subset \text{Gp}(2m,\mathbb{C})$.\\
		The dimension of the tangent space of these deformations is $h^1(P^H(\mathfrak{H})) - g$, where $\mathfrak{H}$ is the Lie algebra of $H$ and $P^H(\mathfrak{H})$ is the adjoint bundle. \\
		We will show that $h^0(P^H(\mathfrak{H})) = 1$. First note that \begin{equation}\label{eqn4}
		h^0(P^H(\mathfrak{H})) \geq \dim\mathfrak{z}(\mathfrak{H})=1,
		\end{equation}
		where $\mathfrak{z}(\mathfrak{H}) \subset \mathfrak{H}$ is the center.\\ Since $P$ is a stable parabolic $\text{Gp}(2m,\mathbb{C})$-bundle and the parabolic weights are small enough such that the underlying vector bundle $E$ is also stable, we have
		\[
		H^0(X,P(\mathfrak{gp}(2m,\mathbb{C}))) \subset H^0(X,E(\mathfrak{gl}(2m,\mathbb{C})))= \mathfrak{z}(\mathfrak{gl}(2m,\mathbb{C})).
		\]
		Since $\mathfrak{H}$ is a submodule of the $H$-module $\mathfrak{gp}(2m,\mathbb{C})$, the adjoint bundle $P^H(\mathfrak{H})$ is a subbundle of $P(\mathfrak{gp}(2m,\mathbb{C}))$. So,
		\[
		h^0(P^H(\mathfrak{H})) \leq h^0(P(\mathfrak{gp}(2m,\mathbb{C}))) \leq \dim \mathfrak{z}(\mathfrak{gl}(2m,\mathbb{C})) =1.
		\]
		This inequality together with the inequality in (\ref{eqn4}) gives $h^0(P^H(\mathfrak{H})) = 1$.\\
		By Riemann-Roch theorem, 
		\[h^1(P^H(\mathfrak{H})) - g = -\text{pardeg}(P^H(\mathfrak{H})) + (\text{rank}(P^H(\mathfrak{H})) - 1)(g - 1).
		\]
		Hence by Lemma \ref{lemma1}, the dimension of the subscheme $Z \subset \mathcal{M}_{Sp}(2m,\alpha,L)$ defined by all stable symplectic parabolic bundles which are not $(1,0)$-stable satisfies the inequality
        \[
        \begin{aligned}
        \dim Z & \leq \max_{1\leq r \leq m}\{2m + (2m^2 + m -2mr + \frac{3r^2 - r}{2})(g-1)\}\\
        & = \dim\mathcal{M}_{Sp}(2m,\alpha,L) - m^2n +2m + \max_{1\leq r \leq m}r(-2m + \frac{3r - 1}{2})(g-1)\\
        & \leq \dim\mathcal{M}_{Sp}(2m,\alpha,L) - m^2n +2m - 1\\
        & < \dim\mathcal{M}_{Sp}(2m,\alpha,L).
        \end{aligned} 
        \]
		Here we have used that $m > 1$ and $g \geq 3$. This completes the proof assuming Lemma \ref{lemma1}.
	\end{proof}
	
	\begin{lemma}\label{lemma1}
		Let $(E_*, \psi)$ be a stable symplectic parabolic bundle of rank $2m$, where $\psi : E_* \otimes E_* \to L$. Let $P$ be the parabolic principal $\textnormal{Gp}(2m,\mathbb{C})$-bundle corresponding to $(E_*,\psi)$. Let $F \subset E$ be an isotropic subbundle of rank $r$ that contradicts the $(1,0)$-stability condition for $(E_*,\psi)$. The subbundle $F_*$ gives a reduction of the structure group $P^H \subset P$ to a parabolic subgroup $H \subset \textnormal{Gp}(2m,\mathbb{C})$. Then 
		\[
		\textnormal{pardeg}(P^H(\mathfrak{H})) \geq -2m,
		\]
		where $P^H(\mathfrak{H})$ is the adjoint bundle of $P^H$. Also, $\textnormal{rank}(P^H(\mathfrak{H})) = 2m^2 + m - 2mr + \frac{3r^2-r}{2} + 1$.
	\end{lemma}
	\begin{proof}
		Let $P^H(\mathfrak{gl}(2m,\mathbb{C}))$ be the associated bundle for the adjoint action of $H$ on $\mathfrak{gl}(2m,\mathbb{C})$. So, $P^H(\mathfrak{gl}(2m,\mathbb{C})) = \text{End}(E_*)$. Since $\mathfrak{H}$ is a submodule of the $\mathfrak{gp}(2m,\mathbb{C})$-module $\mathfrak{gl}(2m,\mathbb{C})$, the adjoint bundle $P^H(\mathfrak{H})$ is a subbundle of $P^H(\mathfrak{gl}(2m,\mathbb{C})) = \text{End}(E_*)$. The subbundle $P^H(\mathfrak{H})$ preserves the filtration
		\begin{equation}\label{eqn3}
		F_*\subset F_*^\perp \subset E_*,
		\end{equation}
		where $F_*^\perp$ is the orthogonal part of $F_*$ with respect to the symplectic structure on $E_*$. So $\textnormal{rank}(F_*)+ \textnormal{rank}(F_*^\vee)=2m$.\\
		Let $L(H)$ denote the Levi quotient of $H$. Fixing $T \subset B \subset H$, where $T \subset \text{Gp}(2m,\mathbb{C})$ is a maximal torus and $B \subset \text{Gp}(2m,\mathbb{C})$ is a Borel subgroup, $L(H)$ can be realized as a subgroup of $H$. Indeed, $L(H)$ is identified with the maximal connected $T$-invariant reductive subgroup of $H$.\\
		Using the projection of $H$ to the Levi quotient $L(H)$, we can extend the structure group of $P^H$ to obtain a $L(H)$-bundle, which we denote by $P^{L(H)}$. Let $P^{L(H)}(H)$ denote the $H$-bundle obtained by extending the structure group of $P^{L(H)}$ using the inclusion $L(H)\subset H$. The $H$-bundle $P^H$ is topologically isomorphic to the $H$-bundle $P^{L(H)}(H)$. Therefore their adjoint bundles $P^H(\mathfrak{H})$ and $P^{L(H)}(H)(\mathfrak{H})$ are topologically isomorphic. Hence we can assume that $P^H$ admits a reduction of the structure group $P^{L(H)}\subset P^H$ to the subgroup $L(H) \subset H$. Fix one such reduction of the structure group  $P^{L(H)}\subset P^H$ to $L(H)$.\\
		Using the reduction of the structure group $P^{L(H)} \subset P^H$ the above filtration (\ref{eqn3}) splits, i.e. 
		\begin{equation}\label{eqn1}
		E_* \cong F_* \oplus (F_*^\perp/F_*) \oplus (E_*/F_*^\perp).
		\end{equation}
		Therefore, a locally defined section of $P^H(\mathfrak{H})$ has the form
		\begin{equation}\label{eqn2}
		A = \begin{pmatrix}
		\alpha & \beta & \gamma \\
		0 & \delta & \epsilon \\
		0 & 0 & \eta
		\end{pmatrix}.
		\end{equation}
		The symplectic structure $\psi : E_* \otimes E_* \to L$ induces an isomorphism
		\[
		\psi : E_* \to E_*^\vee \otimes L,
		\]
		which has the property that the composition 
		\[
		F_*^\perp \xhookrightarrow{} E_* \xrightarrow{\psi} E_*^\vee \otimes L \xrightarrow{} F_*^\vee \otimes L
		\]
		vanishes. So, we obtain an induced isomorphism
		\[
		E_*/F_*^\perp \cong F_*^\vee \otimes L,
		\]
		denoted by $id$. Also, $\psi$ induces a symplectic structure on the parabolic bundle $F_*^\perp/F_*$, and this induced symplectic sturcture will produce an isomorphism
		\[
		\psi' : F_*^\perp/F_* \xrightarrow{} (F_*^\perp/F_*)^\vee \otimes L .
		\]
		Using (\ref{eqn1}), the isomorphism $\psi$ is of the form
		\[
		\psi = \begin{pmatrix}
		0 & 0 & -id \\
		0 & \psi' & 0 \\
		id & 0 & 0
		\end{pmatrix}.
		\]
		A parabolic subalgebra $\mathfrak{H}$ of $\mathfrak{gp}(2m,\mathbb{C})$ is of the form $ \mathfrak{H} = \mathfrak{H'}\oplus \mathbb{C}$, where $\mathfrak{H'}$ is a parabolic subalgebra of $\mathfrak{sp}(2m,\mathbb{C})$, and $\mathbb{C}$ is the center of $\mathfrak{gp}(2m,\mathbb{C})$. Since this decomposition is preserved by the adjoint action of $\text{Gp}(2m,\mathbb{C})$, we get that
		\[
		P^H(\mathfrak{H}) = P^H(\mathfrak{H'})\oplus \mathcal{O}_X
		\]
		The local section $A$ of $P^H(\mathfrak{H})$, defined in (\ref{eqn2}), lies in $P^H(\mathfrak{H'})$ if and only if
		\[
		\psi \circ A =  \begin{pmatrix}
		0 & 0 & -\eta \\
		0 & \psi'\circ \delta & \psi' \circ \epsilon \\
		\alpha & \beta & \gamma
		\end{pmatrix} : E_* \to E_*^\vee \otimes L
		\]
		is symmetric, i.e. the following conditions hold:\\
		$(1) \hspace{0.2cm}-\eta = \alpha^t$,\\
		$(2) \hspace{0.2cm}\psi' \circ \epsilon = \beta^t$, and\\
		$(3) \hspace{0.2cm}\psi' \circ \delta$ and $\gamma$ are symmetric. \\
		Hence, we have an isomorphism
		\[
		P^H(\mathfrak{H'}) \cong \text{End}(F_*) \oplus \Bigg(\bigg(\frac{F_*^\perp}{F_*}\bigg)^\vee \otimes F_*\Bigg) \oplus ((\text{Sym}^2F_*) \otimes L^\vee) \oplus \Bigg(\text{Sym}^2\bigg(\frac{F_*^\perp}{F_*}\bigg)^\vee \otimes L\Bigg)
		\]
		defined by
		\[
		A \mapsto (\alpha,\beta, \gamma, \psi'\circ \delta).
		\]
		A straightforward calculation using this isomorphism gives  \[\textnormal{rank}(P^H(\mathfrak{H}))= \textnormal{rank}(P^H(\mathfrak{H'})) + 1 = 2m^2 + m - 2mr + \frac{3r^2-r}{2} + 1.
		\]
		Also,
		\[
		\textnormal{pardeg}(P^H(\mathfrak{H}))=\textnormal{pardeg}(P^H(\mathfrak{H'})) = 2mf - rf - f -rml + \frac{r^2 - r}{2}l,
		\]
		where $l=\deg L$ and $f=\textnormal{pardeg}(F_*)$. Since $F_*$ contradicts $(1,0)$-stability condition, we have $2f \geq lr - 2$, and since $E_*$ is stable, we have $2f < lr$. Using these two inequalities, we obtain \[\textnormal{pardeg}(P^H(\mathfrak{H})) \geq -2m.
		\] 
		This completes the proof of the lemma.
	\end{proof}

\begin{lemma}\label{lemma2}
	Let $(E_*,\psi)$ be a $(1,0)$-stable symplectic parabolic bundle. Let $x \in D$ and $1 \leq k \leq 2m$ be an integer. Suppose $E_{x,k}' \subsetneq E|_x$ such that 
	\[
	E_{x,k-1} \supsetneq E_{x,k}' \supsetneq E_{x,k+1}.
	\]
	 substitute $E_{x,k}$ by $E_{x,k}'$ to obtain a new quasi-parabolic structure. Then the symplectic parabolic bundle $E_{*'}$ with the new quasi-parabolic structure is stable.
\end{lemma}
\begin{proof}
	Let $F \subset E$ be an isotropic subbundle. Let $F_*$ and $F_{*'}$ be the vector bundles with parabolic structures induced by $E_*$ and $E_{*'}$ respectively. Therefore, we have the relation
	\[
	\text{wt}_x(F_{*'})= \text{wt}_x(F_*) + (\alpha_k(x) - \alpha_{k-1}(x))(\dim(F|_x \cap E_{x,k}')- \dim(F|_x \cap E_{x,k})). 
	\]
	Note that $E_{x,k+1}\subseteq E_{x,k}\cap E_{x,k}'$, and $E_{x,k+1}$ has codimension $1$ in both $E_{x,k}$ and  $E_{x,k}'$. So $\dim(F|_x \cap E_{x,k}') \leq \dim(F|_x \cap E_{x,k}) + 1 $, and hence
	\[
	\text{wt}_x(F_{*'}) \leq  \text{wt}_x(F_*) + (\alpha_k(x) - \alpha_{k-1}(x)) < \text{wt}_x(F_*) + 1.
	\]
	By the $(1,0)$-semistablity condition, we get 
	\[
\begin{aligned}
\frac{\text{pardeg}(F_{*'})}{\text{rank}(F)} &= \frac{\deg(F)+\sum_{x \in D}\text{wt}_x(F_{*'})}{\text{rank}(F)}\\
& < \frac{\deg(F)+\sum_{x \in D}\text{wt}_x(F_{*})+1}{\text{rank}(F)}\\
&= \frac{\text{pardeg}(F_{*})+1}{\text{rank}(F)}\\
&< \frac{\text{pardeg}(E_{*})}{\text{rank}(E)}\\
&= \frac{\text{pardeg}(E_{*'})}{\text{rank}(E)}
\end{aligned}
\]
Since this strict inequality holds for any isotropic subbundle $F \subset E$, we conclude that $E_{*'}$ is stable.
\end{proof}
We also need the following lemma for our purpose.
	
    \begin{lemma}\label{lemma3}
    	Suppose that $g \geq 4$ and the parabolic weights are small enough so that the stability of the symplectic parabolic Higgs bundle is equivalent to the stability of the underlying vector bundle. Then $H^0(\textnormal{End}(E_*)(x))=0$ for a generic stable symplectic parabolic bundle $E_*$.
   \end{lemma}	
	\begin{proof}
 	Let $f : \mathcal{M}_{\text{Sp}}(2m,\alpha,L) \to \mathcal{M}(2m)$ be the forgetful morphism sending a symplectic parabolic bundle to the underlying vector bundle. This morphism is well defined because of the small weights. For $g \geq 4$, there exists an open set $U \in \mathcal{M}(2m)$ such that for every $E \in U$ we have $H^0(\textnormal{End}(E)(x))=0$ \cite[Lemma 2.2]{BGM13}. Hence, for every $E_* \in f^{-1}(U)$, we have $H^0(\textnormal{End}(E_*)(x)) =0$.
	\end{proof}	
	
	\section{Hitchin discriminant}
	We will now describe the Hitchin map for symplectic parabolic Higgs bundles. An element of $\mathcal{N}_{\text{Sp}}(2m,\alpha,L)$ can be viewed as a stable parabolic bundle $E_*$ of rank $2m$ with a nondegenerate symplectic form $< , >$, together with a holomorphic section $\Phi \in H^0(X, \operatorname{End}(E_*) \otimes K(D))$ which satisfies
	\[
	<\Phi v,w> = - <v,\Phi w>.
	\]
	From the nondegeneracy of the symplectic form it follows that if $\mu$ is an eigenvalue then $-\mu$ is also an eigenvalue. Thus the characteristic polynomial of $\Phi$ is of the form
	\[
	\det(\lambda - \Phi) = \lambda^{2m} + s_2\lambda^{2m-2} + \cdots + s_{2m},
	\]
	where $s_{2i} \in H^0(X, K(D)^{2i})$ for all $1\leq i \leq m$. In fact, since $\Phi$ is strongly parabolic, the residue at each point of $D$ is nilpotent. Therefore, $s_{2i} \in H^0(X, K^{2i}D^{2i-1})$ for all $1 \leq i \leq m$. We define the \textit{Hitchin space} as
	\[
	W =  H^0(K^2D) \oplus H^0(K^4D^3) \oplus \dots \oplus H^0(K^{2m}D^{2m-1}).
	\]
Its dimension is $m(2m+1)(g-1) + m^2n$. The \textit{Hitchin map} 
	\[
	h : \mathcal{N}_{\textnormal{Sp}}(2m,\alpha,L) \longrightarrow W
	\]
	is defined by  $h(E_*,\Phi,\psi) = (s_2,s_4,\dots,s_{2m})$ where $s_{2i}$'s are the coefficients of the characteristic polynomial of $\Phi$. Also, we have the restriction map 
	\[
	h_0 : T^*\mathcal{M}_{\textnormal{Sp}}(2m,\alpha,L) \longrightarrow W.
	\]
	Let $\mathcal{S}$ denote the total space of  $K(D)$, let $p : \mathcal{S} \to X$ be the natural projection map, and let $t \in H^0(\mathcal{S},p^*K(D))$ be the tautological section. Given $s=(s_2,\dots, s_{2m}) \in W$, the \textit{spectral curve} $X_s$ in $\mathcal{S}$ is defined by the equation
	\[
	t^{2m} + s_2 t^{2m-2} + \cdots + s_{2m}=0,
	\]
	and it possesses an involution $\sigma(\eta) = -\eta $ (since all the exponents of $t$ are even). The involution $\sigma$ acts on the Jacobian $J(X_s)$ of the spectral curve. If $X_s$ is smooth, then the fiber $h^{-1}(s)$ is isomorphic to the Prym variety $\text{Prym}(X_s, X)=\{M \in J(X_s) : \sigma^*M \cong M^\vee \}$ \cite[Theorem 4.1]{R20}.\\
	Let $\mathcal{D} \subset W$ be the divisor consisting of the characteristic polynomials whose corresponding spectral curve is singular. The inverse image $h^{-1}(\mathcal{D})$ is called the \textit{Hitchin discriminant}. If a spectral curve is singular over a parabolic point $x \in D$, it is singular precisely at $(x,0)$. Consider the following subsets of $\mathcal{D}$:\\
	$(a)$ For each parabolic point $x \in D$, let $\mathcal{D}_x$ be the set of points whose spectral curve is singular over $x$.\\
	$(b)$ Let $\mathcal{D}_1$ be the set of points whose spectral curve is smooth over every $x\in D$ but singular at $(y,0)$, where $y \notin D$.\\
	$(c)$ Let $\mathcal{D}_2$ be the set of points whose spectral curve has two symmetrical nodes (i.e. $t^m + s_2t^{m-1} + \cdots + s_{2m-2}t + s_{2m} = 0$ has a node).\\
	Therefore, $\mathcal{D} = \underset{x \in D}{\bigcup}\mathcal{D}_x \cup \overline{\mathcal{D}_1} \cup \mathcal{D}_2$, where $\overline{\mathcal{D}_1}$ is the set of points whose spectral curve is singular at $(y,0)$ where $y \notin D$ (but not necessarily smooth over $D$).\\
	Let $\mathcal{D}_i^\mathrm{o} \subset \mathcal{D}_i, i=1,2$, denote the locus of the spectral curves which do not contain extra singularities.
	\begin{lemma}\label{lemma4}
		$\mathcal{D}_1$ and $\mathcal{D}_x$ are irreducible for every $x \in D$. 
	\end{lemma}
\begin{proof}
	  Suppose $s \in \mathcal{D}_x$, then the spectral curve is singular at $(x,0)$, i.e. $t=0$ and $s_{2m}$ has a double root at $x$. So $s_{2m} \in H^0(K^{2m}D^{2m-1}(-x))$. Therefore,
	 \[
	 \mathcal{D}_x = \bigoplus\limits_{i=1}^{m-1}H^0(K^{2i}D^{2i-1}) \oplus H^0(K^{2m}D^{2m-1}(-x))
	 \]
	 is irreducible for all $x \in D$. \\
	 Consider $s\in \mathcal{D}_1$, i.e. the spectral curve is singular at $(y,0)$, where $y \notin D$ but smooth over every $x \in D$. So
	 \[
	 s_{2m} \in H^0(K^{2m}D^{2m-1}(-2y)) \setminus \bigcup_{x\in D} H^0(K^{2m}D^{2m-1}(-2y -x)) \coloneqq \mathcal{H}_y
	 \]
	 Hence, we get that
	 \[
	 \mathcal{D}_1 = \bigoplus\limits_{i=1}^{m-1}H^0(K^{2i}D^{2i-1}) \oplus \mathcal{H}_y.
	 \]
	 By Riemann-Roch theorem, the last summand $\mathcal{H}_y$ is the complement of an hyperplane in $H^0(K^{2m}D^{2m-1}(-2y))$. Therefore, $\mathcal{D}_1$ is irreducible. 
	\end{proof}

	\begin{prop}\label{prop2}
		Let $Y$ be an integral curve which has a unique simple node. Also assume that $Y$ possess an involution $\sigma$. Let $\pi_Y : \tilde{Y} \to Y$ be the normalization. Then the compactified Jacobian $\bar{J}(Y)$ is birational to a $\mathbb{P}^1$-fibration over $J(\Tilde{Y})$. \\
		Analogously, let $Y$ be an integral curve with two simple nodes and possess an involution $\sigma$ which interchanges these two nodes, and let $\tilde{Y}$ be the normalization of $Y$. Then $\bar{J}(Y)$ is birational to a $\mathbb{P}^1 \times \mathbb{P}^1$-bundle on $J(\tilde{Y})$.\\
		And in either case the Prym variety, which is the fixed point variety of the involution, is an uniruled variety.
	\end{prop}
\begin{proof}
The proof is same as in \cite{BGM12}. For convenience of the reader, we give details of the proof. Suppose $Y$ has a simple node at $y$, and let $x$ and $z$ be the preimages of $y$ in $\tilde{Y}$. Let $P$ be a $\mathbb{P}^1$-fibration over $J(\tilde{Y})$, whose fiber over any point $L \in J(\tilde{Y})$ is $\mathbb{P}^1(L_{x} \oplus L_{z})$. We claim that the compactified Jacobian $\bar{J}(Y)$ is birational to $P$. Let us construct a morphism $P \to \bar{J}(Y)$ as follows: a point of $P$ corresponds to a line bundle $L \in  J(\tilde{Y})$ and an one dimensional quotient $q : L_{x} \oplus L_{z} \twoheadrightarrow \mathbb{C}$ (up to a scalar multiple). The image $L'$ of $L$ under this morphism is defined as
\[
0 \rightarrow L' \rightarrow (\pi_Y)_*L \overset{q}\rightarrow \mathbb{C}_y \rightarrow 0.
\]
For the details of the proof, see \cite[Theorem 4]{B96}.\\ 
The involution $\sigma$ lifts to an involution $\tilde{\sigma}$ of $\tilde{Y}$, which will induce an involution in $P$ as follows: suppose $(L,q)$ is a point of $P$, and the quotient $q : L_{x} \oplus L_{z} \twoheadrightarrow \mathbb{C}$ is represented by $[a : b]$. Then the induced involution sends the point $(L,q)$ to $(\tilde\sigma^*L^\vee, q^\vee := [b:a])$. The definition of $q^\vee$ makes sense:  if $[a : b] \in \mathbb{P}(L_{x} \oplus L_{z})$, then 
\[
[b : a] \in \mathbb{P}(L_{z} \oplus L_{x}) = \mathbb{P}(L_{x}^\vee \oplus L_{z}^\vee).
\]
Therefore, the involution on $P$ induces an involution in $\bar{J}(Y)$, which restricts to $M \mapsto \sigma^*M^\vee$ on $J(Y)$. The Prym variety $\textnormal{Prym}(Y,\sigma) = \{M \in J(Y): \sigma^*M\cong M^\vee\} \subset \bar{J}(Y)$ is an uniruled variety since we have a surjective morphism from the $\mathbb{P}^1$-fibration $P|_{\text{Prym}}$ defined by the pullback
	\[
	\begin{tikzcd}
	P|_{\text{Prym}} \arrow{r}{} \arrow[swap]{d}{} & P \arrow{d}{} \\%
	\textnormal{Prym}(\tilde{Y},\tilde{\sigma}) \arrow{r}{}& J(\tilde{Y})
	\end{tikzcd}
	\]
Analogously, if $Y$ is an integral curve with two simple nodes $y_1$ and $y_2$ and $\pi_Y^{-1}(y_i) = \{x_i,z_i\}$, then $\bar{J}(Y)$ is birational to a $\mathbb{P}^1 \times \mathbb{P}^1$-bundle $P$ on $J(\tilde{Y})$, where the two $\mathbb{P}^1$-factors correspond to the quotients $q_1 : L_{x_1} \oplus L_{z_1} \twoheadrightarrow \mathbb{C}$ and $q_2 : L_{x_2} \oplus L_{z_2} \twoheadrightarrow \mathbb{C}$.\\
The induced involution on $P$ sends $(L,q_1,q_2)$ to $(\tilde{\sigma}^*L^\vee, q_2^\vee, q_1^\vee)$, and then induces an involution in $\bar{J}(Y)$. A fixed point $(L,q_1,q_2)$ in $P$ for this involution has $L \cong \tilde{\sigma}^*L^\vee$ and $q_2 = q_1^\vee$, therefore the fixed point variety is a $\mathbb{P}^1$-fibration on $\textnormal{Prym}(\tilde{Y},\tilde{\sigma})$, and the image under the birational map is the fixed point locus on $\bar{J}(Y)$, denoted by $\textnormal{Prym}(Y,\sigma)$. So this Prym variety is also an uniruled variety.  
\end{proof}
	
	\begin{prop}\label{prop3}
		Let $g \geq 4$. The following statements hold for the Hitchin map $h_0 : T^*\mathcal{M}_{\textnormal{Sp}}(2m,\alpha,L) \to W$: \\
		(a) For $s \in W -\mathcal{D}$, the fiber $h_0^{-1}(s)$ is an open subset of an abelian variety.\\
		(b) For $s \in \mathcal{D}_i^\mathrm{o}$  $(i = 1,2)$, the fiber  $h_0^{-1}(s)$ is an open subset of an uniruled variety.\\
		(c) For a generic $s \in \mathcal{D}_x$, the fiber  $h_0^{-1}(s)$ contains a complete rational curve.
	\end{prop}
	\begin{proof}
		The Hitchin map $h : \mathcal{N}_{\text{Sp}}(2m,\alpha,L) \to W$ is proper (see \cite{M94} for details). By \cite[Theorem, 4.1]{R20}, $h^{-1}(s)$ is an open subset of an abelian variety (a Prym variety) for $s \in W - \mathcal{D}$.\\
		The complement $\mathcal{N}_{\text{Sp}}(2m,\alpha,L) - T^*\mathcal{M}_{\text{Sp}}(2m,\alpha,L)$ has codimension $\geq 3$ (following the computations in Faltings \cite[Theorem II.6(iii)]{F93}, for $g \geq 4$). Therefore $(\mathcal{N}_{\text{Sp}}(2m,\alpha,L) - T^*\mathcal{M}_{\text{Sp}}(2m,\alpha,L)) \cap h^{-1}(\mathcal{D}_i)$ is of codimension at least $2$ in $h^{-1}(\mathcal{D}_i)$, so for $s\in \mathcal{D}_i^\mathrm{o}$, we get that
		\[
		h^{-1}(s) - h_0^{-1}(s) \subset h^{-1}(s)
		\]
		has codimension at least two. Therefore, by Proposition \ref{prop2}, $h_0^{-1}(s)$ is an open subset of an uniruled variety.\\
		It remains to prove part $(c)$, that a generic fiber over $\mathcal{D}_x$ contains a complete rational curve. The idea of the proof is same as in \cite[Proposition 4.2]{AG19}. Let $V \subset \mathcal{M}_{\textnormal{Sp}}(2m,\alpha,L)$ denote the intersection of open subsets defined by Proposition \ref{prop1} and Lemma \ref{lemma3}, i.e. $V$ consists of $(1,0)$-stable symplectic parabolic bundles $(E_*,\psi)$ such that $H^0(\textnormal{End}(E_*)(x))=0$. Thus, for all $(E_*,\psi) \in V$ and all $x \in D$, we have
     \[
     H^1(\textnormal{End}(E_*)\otimes K(D-x))= H^0(\textnormal{End}(E_*)(x))^\vee = 0.
     \]
		Therefore the evaluation map
		\[
		\text{ev} : H^0(\textnormal{End}(E_*)\otimes K(D)) \to \textnormal{End}(E_*)\otimes K(D)|_x
		\]
		is surjective.\\
		For $1 < k \leq 2m$, consider the subspace $N_k(E_*) \subset \textnormal{End}(E_*)\otimes K(D)|_x$ whose elements are matrices with a zero at $(k-1,k)$. For $k=1$, the subspace $N_1(E_*)$ consists of matrices with a zero at $(2m,1)$. Let $\tilde{N}_k(E_*)$ be the preimage of $N_k(E_*)$ under the evaluation map. For $k>1$, we can describe $\tilde{N}_k(E_*)$ as follows : consider the subfiltraion of $E|_x$ obtained by removing the subspace $E_{x,k}$. Denote the parabolic bundle with this new filtration by $E_{*_k}$. Then
		\[
		\tilde{N}_k(E_*) = H^0(\textnormal{End}(E_{*_k}) \otimes K(D)).
		\]
	 Let $(E_*,\Phi,\psi) \in h^{-1}(\mathcal{D}_x)\cap T^*V$. We can describe the Higgs field $\Phi$ in a basis corresponding to the parabolic filtration as 
	 \[
	 \Phi(z) = \begin{bmatrix} za_{1,1} & a_{1,2} & \cdots & a_{1,2m} \\ za_{2,1} & za_{2,2} & \cdots & a_{2,2m} \\ \vdots & \vdots & \ddots & \vdots \\ za_{2m,1} & za_{2m,2} & \cdots & za_{2m,2m} \end{bmatrix}
	 \]
	 where $a_{i,j}$ are local sections of $K(D)$. Then $(E_*,\Phi,\psi) \in h^{-1}(\mathcal{D}_x)$ is and only if $z^2 | \det(\Phi(z))$. The only summand of the determinant that is not a multiple of $z^2$ is $za_{2m,1}a_{1,2}a_{2,3}\cdots a_{2m-1,2m}$. Therefore, $z^2|\det(\Phi(z))$ if and only if $\text{ev}(\Phi) \in N_k(E_*)$ for some $1< k \leq 2m$. Since the evaluation map is surjective we conclude that
		\[
		h^{-1}(\mathcal{D}_x) \cap T_{E_*}^*\mathcal{M}_{\textnormal{Sp}}(2m,\alpha,L) = \bigcup_{k=1}^{2m}\tilde{N}_k(E_*)
		\]
		By Lemma \ref{lemma2}, for all $E_* \in V$, for all $x \in D$, all $1< k \leq 2m$ and all $E_{x,k}'$ such that  $E_{x,k-1} \supsetneq E_{x,k}' \supsetneq E_{x,k+1}$, we know that $E_{*'}$ is stable. Since $\Phi$ sends $E_{x,k-1}$ to $E_{x,k+1}$, we have that $\Phi \in H^0(\textnormal{End}(E_{*'})\otimes K(D))$ for all $E_{x,k}'$. Therefore $E_{*'} \in h_0^{-1}(\mathcal{D}_x)$ for all such $E_{x,k}'$. Since $E$ and $\Phi$ remain the same, all those Higgs bundles lie over the same point. The space of possible compatible steps form a complete rational curve.
	\end{proof}

	\begin{prop}\label{prop4}
		Let $\mathcal{C}\subset T^*\mathcal{M}_{\textnormal{Sp}}(2m,\alpha,L)$ be the union of the (complete) rational curves in $T^*\mathcal{M}_{\textnormal{Sp}}(2m,\alpha,L)$. Then $\mathcal{D}$ is the closure of $h_0(\mathcal{C}) \subset W$. 
	\end{prop}
	\begin{proof}
		Let $l \cong \mathbb{P}^1 \subset T^*\mathcal{M}_{\text{Sp}}(2m,\alpha,L)$ be a complete rational curve. So $h_0(l) \subset W$ is a point as it is a complete curve. Therefore $l$ is contained in a fiber of the Hitchin map. By Proposition \ref{prop3}, fiber over $s \in W - \mathcal{D}$ is an open subset of an abelian variety, so $l$ cannnot be contained in a fiber over $W-\mathcal{D}$.\\
		Again by Proposition \ref{prop3}, we know that generic fiber over $\mathcal{D}_x$ and $\mathcal{D}_i^\mathrm{o},i= 1,2$, contain a complete rational curve.  Therefore, $h_0(\mathcal{C})$ is dense in $\mathcal{D}$. Since $\mathcal{D}\subset W$ is closed, we get that $\mathcal{D}= \overline{h_0(\mathcal{C})}$.
	\end{proof}
\section{Torelli theorem for symplectic parabolic bundles}
	\begin{prop}\label{prop5}
		The global algebraic functions $\Gamma(T^*\mathcal{M}_{\textnormal{Sp}}(2m,\alpha,L))$ produce a map
		\[
		\tilde{h} : T^*\mathcal{M}_{\textnormal{Sp}}(2m,\alpha,L) \longrightarrow \mathrm{Spec}(\Gamma(T^*\mathcal{M}_{\textnormal{Sp}}(2m,\alpha,L))) \cong W \cong \mathbb{C}^N,
		\]
		which is the Hitchin map upto an automorphism of $\mathbb{C}^N$, where $N=\dim W =\dim\mathcal{M}_{\textnormal{Sp}}(2m,\alpha,L)$.
		Moreover, if we consider the standard dilation action of $\mathbb{C}^*$ on the fibers of the cotangent bundle $T^*\mathcal{M}_{\textnormal{Sp}}(2m,\alpha,L)$, then there is a unique $\mathbb{C}^*$-action on $W$ such that $\tilde{h}$ is $\mathbb{C}^*$-equivariant, i.e. $\tilde{h}(E_*, \lambda\Phi)= \lambda \cdot \tilde{h}(E_*,\Phi)$.
	\end{prop}
	\begin{proof}
		The Hitchin map
		\[
		h : \mathcal{N}_{\textnormal{Sp}}(2m,\alpha,L) \longrightarrow W
		\]
		is proper and the generic fibers are compact and connected - a Prym variety. Since $W$ is smooth, every fibre is compact and connected. Therefore any holomorphic function $f : \mathcal{N}_{\textnormal{Sp}}(2m,\alpha,L) \to \mathbb{C}$ is constant on each fiber and so $f$ factors through $W$ and, as $W$ is affine, we get that
		\[
		\mathrm{Spec}(\Gamma(\mathcal{N}_{\textnormal{Sp}}(2m,\alpha,L))) \cong \mathrm{Spec}(\Gamma(W)) \cong W.
		\]
		Let  $g : T^*\mathcal{M}_{\textnormal{Sp}}(2m,\alpha,L) \to \mathbb{C}$ be a holomorphic function. Since the codimension of $T^*\mathcal{M}_{\textnormal{Sp}}(2m,\alpha,L) \subset \mathcal{N}_{\textnormal{Sp}}(2m,\alpha,L)$ is at least two, and $\mathcal{N}_{\textnormal{Sp}}(2m,\alpha,L)$ is smooth, by Hartog's theorem $g$ extends to a holomorphic function $\tilde{g} : \mathcal{N}_{\textnormal{Sp}}(2m,\alpha,L) \to \mathbb{C}$. Therefore, we get \[\Gamma(T^*\mathcal{M}_{\textnormal{Sp}}(2m,\alpha,L)) \cong \Gamma(\mathcal{N}_{\textnormal{Sp}}(2m,\alpha,L)).
		\] 
		Hence, we obtain a map
		\[
		\tilde{h} : T^*\mathcal{M}_{\textnormal{Sp}}(2m,\alpha,L) \longrightarrow \mathrm{Spec}(\Gamma(T^*\mathcal{M}_{\textnormal{Sp}}(2m,\alpha,L))) \cong W,
		\]
		and the $\mathbb{C}^*$- action on the cotangent bundle $T^*\mathcal{M}_{\textnormal{Sp}}(2m,\alpha,L)$ then induces a unique action on $\mathrm{Spec}(\Gamma(T^*\mathcal{M}_{\textnormal{Sp}}(2m,\alpha,L)))$ making $\tilde{h}$ a $\mathbb{C}^*$-equivariant map.  
	\end{proof}
	The Proposition \ref{prop5} allow us to recover the base $W$ of the Hitchin fibration, and gives us the natural $\mathbb{C}^*$- action on $W$. Also the subspaces $W_{\geq 2k} = \bigoplus_{i=k}^{m}W_{2i}$ (where $W_{2i} = H^0(K(D)^{2i})$) are uniquely determined for each $k=1,\dots, m$ (these are the subspaces where the rate of decay is at least $\lambda^{2k}$, for $\lambda \to 0$). In particular, we can recover uniquely the subspace $W_{2m} = H^0(K(D)^{2m}) \subset W$.
	\begin{prop}\label{prop6}
		The intersection $\mathcal{C} := W_{2m} \cap \mathcal{D} \subset W_{2m}$ has $n+1$ irreducible components
		\[
		\mathcal{C} = \mathcal{C}_X \cup \bigcup_{x \in D} \mathcal{C}_x.
		\]
		Moreover $\mathbb{P}(\mathcal{C}_X ) \subset \mathbb{P}(W_{2m})$ is the dual variety of $X \subset \mathbb{P}(W_{2m}^*)$ and for each $x \in D$, $\mathbb{P}(\mathcal{C}_x) \subset \mathbb{P}(W_{2m})$ is the dual variety of $x \xhookrightarrow{} X \subset \mathbb{P}(W_{2m}^*)$ for the embedding given by the linear series $|K^{2m}D^{2m-1}|$. 
	\end{prop}
	\begin{proof}
		Let $s = s_{2m} \in H^0(K^{2m}D^{2m-1}) \subset H^0(K(D)^{2m})$. Then the spectral curve $X_s$ is given by the equation $t^{2m} + s_{2m}(y)=0$. Therefore, $X_s$ is singular at the points $(x,0)$, where $x$ is zero of order at least two of $s_{2m}$. Therefore $s \in \mathcal{C}$ if and only if $s_{2m} \in H^0(K(D)^{2m}(-2x))$, by considering $s_{2m}$ as a section of $H^0(K(D)^{2m})$. Since $s_{2m} \in H^0(K^{2m}D^{2m-1})$ the following are the possible cases:\\
		$(a)$ \hspace{0.0001cm}$s_{2m} \in H^0(K^{2m}D^{2m-1}(-2x))$ for some $x \notin D$\\
		$(b)$ \hspace{0.0001cm}$s_{2m} \in H^0(K^{2m}D^{2m-1}(-x))$ for some $x \in D$.\\
		So, we can write
		\begin{align*}
		\mathcal{C} &= \bigcup_{x\in X \setminus D} H^0(K^{2m}D^{2m-1}(-2x)) \cup \bigcup_{x \in D} H^0(K^{2m}D^{2m-1}(-x)) \\
		&= \bigcup_{x\in X} H^0(K^{2m}D^{2m-1}(-2x)) \cup \bigcup_{x \in D} H^0(K^{2m}D^{2m-1}(-x)).
		\end{align*}
		Denote 
		\begin{align*}
		\mathcal{C}_X &= \bigcup_{x\in X} H^0(K^{2m}D^{2m-1}(-2x))\\
		\mathcal{C}_x &= H^0(K^{2m}D^{2m-1}(-x)) \hspace{0.3cm} x \in D
		\end{align*}
		By Lemma \ref{lemma4}, we know that $\mathcal{C}_X$ and, $\mathcal{C}_x$ are irreducible for all $x \in D$ and both have codimension $1$ in $W_{2m}$. Therefore, the first statement follows.\\
		The linear system $|K^{2m}D^{2m-1}|$ is very ample and induces an embedding $X \subset \mathbb{P}(W_{2m}^*)$. The set of hyperplanes in $\mathbb{P}(W_{2m}^*)$ which are tangent to $X$ at $x$ is precisely $\mathbb{P}(H^0(K^{2m}D^{2m-1}(-2x)))$. Therefore, $\mathbb{P}(\mathcal{C}_X ) \subset \mathbb{P}(W_{2m})$ is the dual variety of $X$. Furthermore, for every $x \in D$, the set of hyperplanes in $\mathbb{P}(W_{2m}^*)$  passing through $x$ is $\mathbb{P}(H^0(K^{2m}D^{2m-1}(-x)))$. Therefore, $\mathbb{P}(\mathcal{C}_x) \subset \mathbb{P}(W_{2m})$ is the dual variety of $x \xhookrightarrow{} X \subset \mathbb{P}(W_{2m}^*)$. 
	\end{proof}
	Note that $\mathbb{P}(\mathcal{C}_x) \not \cong \mathbb{P}(\mathcal{C}_X)$, as $\mathbb{P}(\mathcal{C}_x)$ is the dual variety of a point and $\mathbb{P}(\mathcal{C}_X)$ is the dual variety of a compact Riemann surface. Also $\mathcal{C}_x \subset W_{2m}$ is an hyperplane for all $x \in D$. So $\mathcal{C}_X \subset \mathcal{C}$ is the only irreducible component which is not an hyperplane in $W_{2m}$.
	\begin{theorem}\label{thm1}
		Let $(X,D)$ and $(X',D')$ be two compact Riemann surfaces of genus $g \geq 4$ with set of marked points $D \subset X$ and $D' \subset X'$. Let $\mathcal{M}_{\textnormal{Sp}}(2m,\alpha,L)$ and $\mathcal{M}'_{\textnormal{Sp}}(2m,\alpha,L)$ be the moduli spaces of stable symplecic parabolic bundles over $X$ and $X'$ respectively. If $\mathcal{M}_{\textnormal{Sp}}(2m,\alpha,L)$ is isomorphic to $\mathcal{M}'_{\textnormal{Sp}}(2m,\alpha,L)$, then $(X,D)$ is isomorphic to $(X',D')$, i.e. there exist an isomorphism $X \cong X'$ sending $D$ to $D'$. 
	\end{theorem}
	\begin{proof}
		Suppose $\Psi : \mathcal{M}_{\textnormal{Sp}}(2m,\alpha,L) \longrightarrow \mathcal{M}'_{\textnormal{Sp}}(2m,\alpha,L)$ is an isomorphism. Then there is an induced isomorphism $d(\Psi^{-1}) : T^*\mathcal{M}_{\textnormal{Sp}}(2m,\alpha,L) \longrightarrow T^*\mathcal{M}'_{\textnormal{Sp}}(2m,\alpha,L)$, which is $\mathbb{C}^*$-equivariant for the standard dilation action. By Proposition \ref{prop5}, there exist unique $\mathbb{C}^*$- actions on $W$ and $W'$ induced from the $\mathbb{C}^*$-action by dilations on the fibers. Therefore, we have the following commutative diagram
		\[
		\begin{tikzcd}
		T^*\mathcal{M}_{\textnormal{Sp}}(2m,\alpha,L) \arrow{r}{d(\Psi^{-1})} \arrow[swap]{d}{} & T^*\mathcal{M}'_{\textnormal{Sp}}(2m,\alpha,L) \arrow{d}{} \\%
		W \arrow{r}{f}& W'
		\end{tikzcd}
		\]
		for some $\mathbb{C}^*$-equivariant isomorphism $f : W \longrightarrow W'$. Hence, $f$ sends the subspace $W_{2m} \subset W$ to the subspace $W_{2m}' \subset W'$. The restriction map $f : W_{2m} \longrightarrow W_{2m}'$ is $\mathbb{C}^*$-equivariant and homogeneous of degree $2m$, so it is linear. \\
		Since $d(\Psi^{-1})$ is an isomorphism, it maps the complete rational curves to the complete rational curves. By Proposition \ref{prop4}, the locus of singular spectral curves is preserved by $f$, i.e. $f$ sends $\mathcal{D} \subset W$ to $\mathcal{D}' \subset W'$. So the restriction map sends $\mathcal{C}= \mathcal{D} \cap W_{2m}$ to $\mathcal{C}' = \mathcal{D}' \cap W_{2m}'$. This induces an isomorphism $f^\vee : \mathbb{P}(W_{2m}^*) \longrightarrow \mathbb{P}((W_{2m}^{'})^*)$. Since $\mathcal{C}_X \subset \mathcal{C}$ is canonically identified as the only irreducible component which is not an hyperplane, by Proposition \ref{prop6} $f^\vee$ sends $X$ to $X'$. Moreover, again by Proposition \ref{prop6} the divisor $D \subset X$ is the dual of the rest of the components $\mathbb{P}(\mathcal{C}_x) \subset \mathbb{P}(\mathcal{C})$. Therefore, $f^\vee$ must send $D$ to $D'$. Hence, an isomorphism $f^\vee : (X,D) \longrightarrow (X',D')$ is obtained.
	\end{proof}
	\section{The Nilpotent Cone}
	The fiber $h^{-1}(0)$ is called the nilpotent cone. It is Lagrangian (see \cite{G01}, the proof also works for the parabolic case). A symplectic parabolic Higgs bundle is in the nilpotent cone if and only if the Higgs field is a nilpotent endomorphism. The moduli space of symplectic parabolic bundle is embedded in the nilpotent cone as a component and we will see that it is the unique component of the nilpotent cone that doesn't admit a non-trivial $\mathbb{C}^*$-action. A non-trivial $\mathbb{C}^*$-action gives a non-zero vector field, so it is enough to show that $H^0(\mathcal{M}_{\textnormal{Sp}}(2m,\alpha,L),T_{\mathcal{M}_{\textnormal{Sp}}(2m,\alpha,L)})=0$.
	\begin{prop}\label{prop7}
		$\mathcal{M}_{\textnormal{Sp}}(2m,\alpha,L)$ doesn't admit a non-zero vector field, i.e.\\ $H^0(\mathcal{M}_{\textnormal{Sp}}(2m,\alpha,L),T_{\mathcal{M}_{\textnormal{Sp}}(2m,\alpha,L)})=0$.
	\end{prop}
    \begin{proof}
	Let $s \in H^0(\mathcal{M}_{\textnormal{Sp}}(2m,\alpha,L),T_{\mathcal{M}_{\textnormal{Sp}}(2m,\alpha,L)})$. It gives a holomorphic function $\tilde{s}$ on $T^*\mathcal{M}_{\textnormal{Sp}}(2m,\alpha,L)$. By Hartog's theorem $\tilde{s}$ extends to a holomorphic function on $\mathcal{N}_{\textnormal{Sp}}(2m,\alpha,L)$. Recall that the generic fiber of the Hitchin map $h$ is compact and connected. Since $h$ is proper and $W$ is smooth, every fiber is compact and connected. Therefore $\tilde{s}$ is constant on each fiber and hence
	\[
	\tilde{s} = f \circ h
	\]
	for some holomorphic function $f : W \to \mathbb{C}$.\\
	Now consider the standard $\mathbb{C}^*$-action (induced by the map $\Phi \to \lambda\Phi$) on $\mathcal{N}_{\textnormal{Sp}}(2m,\alpha,L)$. Since $s$ is a vector field, we have 
	\[
	\tilde{s}(\lambda \cdot E_*)= \lambda\tilde{s}(E_*) \hspace{0.3cm} \text{for} \hspace{0.1cm} E_* \in \mathcal{N}_{\textnormal{Sp}}(2m,\alpha,L)
	\]
	If $h(E_*) = (s_2, \dots , s_{2m})$, then $h(\lambda \cdot E_*) = (\lambda^2s_2, \dots, \lambda^{2m}s_{2m}) $. Therefore $f$ satisfies
	\begin{equation}\label{eqn5}
	f(\lambda^2s_2, \dots, \lambda^{2m}s_{2m})= \lambda f(s_2,\dots, s_{2m}).
	\end{equation}
	Since there is no nonzero holomorphic function $f : W \to \mathbb{C}$ which satisfies the condition (\ref{eqn5}), we conclude that $\tilde{s} = 0$.
	\end{proof}

We can modify Simpson's Lemma $11.9$ in \cite{S95} to the symplectic parabolic case with nonzero degree, and hence we get the following
\begin{lemma}\label{lemma5}
	Let $(E_*,\Phi)$ be a symplectic parabolic Higgs bundle in the nilpotent cone, with $\Phi \not= 0$. Consider the standard $\mathbb{C}^*$-action sending $(E_*,\Phi,\psi)$ to $(E_*,t\Phi)$. Assume that $(E_*,\Phi)$ is a fixed point, i.e. for every $t$ there is an isomorphism with $(E_*,t\Phi)$. Then there is another Higgs bundle $(E_*',\Phi')$ in the nilpotent cone, not isomorphic to $(E_*,\Phi)$ such that $\lim_{t \to \infty}(E_*',t\Phi')=(E_*,\Phi)$.  
\end{lemma}
Therefore we obtain the following
\begin{prop}\label{prop8}
	There is only one component inside the nilpotent cone that does not admit a nontrivial $\mathbb{C}^*$-action, and it is the moduli space of symplectic parabolic bundles $\mathcal{M}_{\textnormal{Sp}}(2m,\alpha,L)$ on $X$.
\end{prop}
\begin{proof}
	The map that sends $E_*$ to $(E_*,0)$ defines an embedding of $\mathcal{M}_{\textnormal{Sp}}(2m,\alpha,L)$ in the nilpotent cone. Since both have same dimensions, it is one component of the nilpotent cone. By Proposition \ref{prop7}, we know that $\mathcal{M}_{\textnormal{Sp}}(2m,\alpha,L)$ doesn't admit a non-trivial $\mathbb{C}^*$-action.\\
	The $\mathbb{C}^*$-action in the rest of the components is given by $(E_*,\Phi) \mapsto (E_*,t\Phi)$. This action is nontrivial because of lemma \ref{lemma5}.
	\end{proof}	
\section{Kodaira-Spencer map}
Our next goal is to identify the nilpotent cone among the fibers of the Hitchin map.\\
Let $\mathcal{N}(2m,\alpha,\xi)$ denote the moduli space of parabolic Higgs bundles of rank $2m$ on $X$ with fixed determinant $\xi$. Then there is a Hitchin map 
\[
h' : \mathcal{N}(2m,\alpha,\xi) \to W'
\]
where $W' = H^0(K^2D) \oplus H^0(K^3D^2) \oplus \cdots \oplus H^0(K^{2m}D^{2m-1})$. Let $s \in W'$ be a point such that the spectral curve $X_s$ is smooth. The Kodaira-Spencer map
\[
u' : T_sW' \cong W' \to H^1(X_s,T_{X_s})
\]
sends a vector in the tangent space $T_sW'$ to the corresponding infinitesimal deformation of $X_s$, which is an element of $H^1(X_s,T_{X_s})$. We can also consider the restriction of $u'$ to $W$ obtaining another Kodaira-Spencer map
\[
u : T_sW \cong W \to H^1(X_s,T_{X_s}).
\]
By \cite[Proposition 5.2]{GL11}, we have an exact sequence
\[
0 \to H^0(X,\mathcal{O}_X) \xrightarrow{} T_sW' \xrightarrow{u'} H^1(X_s,T_{X_s}), 
\]
and hence $\dim\text{Ker}(u')=1$. There are some elements in $H^0(X,\mathcal{O}_X)$ which are in the kernel of the restricted Kodaira-Spencer map $u$. Let $\lambda \in \mathbb{C} \cong H^0(X, \mathcal{O}_X)$, and $(x,v) \in X_s$. Then the deformation sending $(x,v)$ to $(x,e^\lambda v)$ doesn't change the isomorphism class of $X_s$. This deformation is produced by the standard $\mathbb{C}^*$-action and it is in the kernel of the restricted Kodaira-Spencer map $u$. Therefore we obtain the following
\begin{prop}\label{prop9}
	The kernel of the Kodaira-Spencer map $u$ has dimension $1$.
\end{prop}
The standard $\mathbb{C}^*$-action on $\mathcal{N}_{\textnormal{Sp}}(2m,\alpha,L)$, sending $(E_*,\Phi)$ to $(E_*,t\Phi)$ induces an action on $W$, whose only fixed point is the origin.
\begin{prop}\label{prop10}
	Let $g : \mathbb{C}^* \times W \to W$ be an action, having exactly one fixed point, and admitting a lift to $\mathcal{N}_{\textnormal{Sp}}(2m,\alpha,L)$. Then this fixed point is the origin of $W$.
\end{prop}
\begin{proof}
	The proof is completely analogous to the proof in \cite[Proposition 5.1]{BG03}. For convenience of the reader, we only give a sketch of the proof.\\
	Let $s \in W$ be a point such that the corresponding spectral curve $X_s$ is smooth. The tangent vector defined at $s$ by the standard $\mathbb{C}^*$-action is contained in the kernel of the Kodaira-Spencer map $u$, because the standard action doesn't change the spectral curve (upto isomorphism). We are going to show that the tangent vector defined by any action which admits a lift to $\mathcal{N}_{\textnormal{Sp}}(2m,\alpha,L)$ is also in the kernel of the Kodaira-Spencer map $u$.\\
	Denote $J=J(X)$, $J_s = J(X_s)$ and $P_s = \text{Prym}(X_s/X)$. There is an unramified covering map 
	\[
	q : J \times P_s \to J_s
	\] 
	sending $(L_1,L_2)$ to $L_1 \otimes \pi^*L_2$, where $\pi : X_s \to X$ is the projection map. Therefore, a deformation of $J_s$ gives a deformation of $J \times P_s$.\\
	Let $\eta$ be the tangent vector defined by the given action $g : \mathbb{C}^* \times W \to W$. The image $u(\eta)$ under the Kodaira-Spencer map produces a deformation of the spectral curve $u(\eta)= \eta_1 \in H^1(X_s,T_{X_s})$, its Jacobian $\eta_2 \in H^1(J_s, T_{J_s})$, and the Prym variety $\eta_3 \in H^1(P_s,T_{P_s})$. A deformation of $X_s$ produces a deformation of $J_s$. We have the following homomorphisms
	\[
	H^1(X_s,T_{X_s}) \xhookrightarrow{i} H^1(J_s,T_{J_s}) \xhookrightarrow{\epsilon} H^1(J\times P_s,T_{J\times P_s}) \xhookleftarrow{} H^1(P_s,T_{P_s})
	\]
	The map $i$ is injective because of the classical Torelli theorem (a nonzero deformation of a curve produces a nonzero deformation of its Jacobian). The image $\epsilon \circ i$ actually lies in $H^1(P_s,T_{P_s})$, since a deformation of $J\times P_s$ induced by a deformation of $X_s$ is induced by a deformation of $P_s$.\\
	The fiber of the Hitchin map $h$ over $s$ is isomorphic to $P_s$. Since the action $g$ lifts to $\mathcal{N}_{\textnormal{Sp}}(2m,\alpha,L)$, the fiber $P_s$ is isomorphic to $P_{g(t,s)}$ at $g(t,s)$ for all $t \in \mathbb{C}^*$. Therefore we have $\eta_3=0$, and hence $\eta_1=0$. So the tangent vector $\eta$ defined by the action $g$ lies in the kernel of the Kodaira-Spencer map $u$. By Proposition \ref{prop9}, we can say that the orbit of $g$ through $s$ is included in the orbit of the standard action through $s$. In particular, since the origin is the fixed point of the standard action, it is a limiting point of the orbit of $g$ through $s$. But the origin is not in the orbit as the fiber over the origin is not isomorphic to the fiber over $s$. Since the limiting points of an orbit are the fixed points, the origin is a fixed point of the action. Since $g$ has exactly one fixed point, the origin is the only fixed point.
	\end{proof}	
\section{Proof of main theorem}

\begin{proof}[Proof of Theorem \ref{thm2}]
	Consider $Y$ an algebraic variety isomorphic to the moduli space $\mathcal{N}_{\textnormal{Sp}}(2m,\alpha,L)$. Since the fibers of the Hitchin map $h$ are compact, the ring of global functions of $\mathcal{N}_{\textnormal{Sp}}(2m,\alpha,L)$ factorizes through the $h$ and therefore
\[
\text{Spec}(\Gamma(Y)) \cong \text{Spec}(\Gamma(W)) \cong \mathbb{C}[y_1,y_2,\dots,y_{N}],
\]
where $N = \dim \mathcal{M}_{\textnormal{Sp}}(2m,\alpha,L) = m(2m+1)(g-1)+m^2n$. Hence, there is an isomorphism $\beta : \mathbb{A}^{N} \to W$ such that the following diagram commutes
\begin{equation}\label{diag1} 
\begin{tikzcd}
Y \arrow[r, "\sim", "\alpha"'] \arrow[swap]{d}{m} & \mathcal{N}_{\textnormal{Sp}}(2m,\alpha,L) \arrow{d}{h} \\%
\mathbb{A}^{N} \arrow[r,"\sim", "\beta"']& W
\end{tikzcd}
\end{equation}
Consider a $\mathbb{C}^*$-action $g : \mathbb{C}^* \times \mathbb{A}^{N} \to \mathbb{A}^{N}$ with exactly one fixed point $v$ and admitting a lift to $Y$. Such an action exists because of the standard $\mathbb{C}^*$-action on $W$ and the isomorphism $\beta$.\\
By Proposition \ref{prop10}, $\beta(v)$ is the origin of $W$. Therefore, the fiber $m^{-1}(v)$ is isomorphic to the nilpotent cone. By Proposition \ref{prop8}, the moduli space $\mathcal{M}_{\textnormal{Sp}}(2m,\alpha,L)$ of stable symplectic parabolic bundles is the unique irreducible component of the nilpotent cone which doesn't admit a nontrivial $\mathbb{C}^*$-action. Therefore, if $\mathcal{N}_{\textnormal{Sp}}(2m,\alpha,L)$ is isomorphic to $\mathcal{N}'_{\textnormal{Sp}}(2m,\alpha,L)$, then $\mathcal{M}_{\textnormal{Sp}}(2m,\alpha,L)$ is isomorphic to $\mathcal{M}'_{\textnormal{Sp}}(2m,\alpha,L)$, and by Theorem \ref{thm1}, $(X,D)$ is isomorphic to $(X',D')$. 
\end{proof}


\begin{thebibliography}{30}
    \bibitem{MN68}
    D. Mumford, P. Newstead, 
    \emph{Periods of a moduli space of bundles on curves}. Amer. J. Math. 90 (1968), 1200–1208.
    
    \bibitem{NR75}
     M.S. Narasimhan, S. Ramanan, 
    \emph{Deformations of the moduli space of vector bundles
    over an algebraic curve}, Ann. Math. (2) 101 (1975), 391–417.

     \bibitem{MS80}
      V.B. Mehta, C.S. Seshadri, 
      \emph{Moduli of vector bundles on curves with parabolic structures}, Math. Ann. 248, (1980) 205–239. 

	\bibitem{BR89}
	U. Bhosle, A. Ramanathan,
	\emph{Moduli of parabolic $G$-bundles on curves},
	Math. Z. 202, no. 2 (1989), 161–180.
	
	\bibitem{B96}
	U. Bhosle,
	\emph{Generalized parabolic bundles and applications–II}, 
	Proc. Indian Acad. Sci.(Math. Sci.) 106 (1996), 403–420.
	
	\bibitem{AG19}
	 D. Alfaya, T. Gómez, 
	 \emph{Automorphism group of the moduli space of parabolic bundles over a curve}, arXiv:1905.12404 (2019).
	
	\bibitem{GL11}
	T. Gómez, M. Logares,
	\emph{A Torelli theorem for the moduli space of parabolic Higgs bundles}, Adv.Geom. 11 (2011), 429–444.
	
	\bibitem{BG03}
	I. Biswas, T. Gómez, 
	 \emph{A Torelli theorem for the moduli space of Higgs bundles on a
	curve}, Quarterly Journal of Mathematics 54 (2003), 159–169.
	
	\bibitem{BG06}
	I. Biswas, T. Gómez, 
	 \emph{Hecke correspondence for symplectic bundles with application
	to the Picard Bundles}, Inter. Jour. Math. 17 (2006), 45–63.
	
	\bibitem{BH12}
	I. Biswas, N. Hoffmann,
	\emph{A Torelli theorem for moduli spaces of principal bundles over a curve}, Ann. Inst. Fourier (Grenoble) 62 (2012) no. 1, p. 87-106
	
	\bibitem{BGM12}
	I. Biswas, T. Gómez, V. Muñoz,
	\emph{Automorphisms of moduli spaces of symplectic bundle},
	Int. J. Math., Vol. 23, No. 5, 1520052 (2012).
	
	\bibitem{BGM13}
	I. Biswas, T. Gómez, V. Muñoz, 
	 \emph{Automorphisms of moduli spaces of vector bundles
	over a curve}, Expo. Math. 31 (2013), 73–86.
	
	\bibitem{BMW11}
	I. Biswas, S. Majumder, M.L. Wong,
	\emph{Orthogonal and symplectic parabolic bundles},
	J. Geom. Phys. 61 (2011), 1462–1475.
	
	 \bibitem{BBN01}
	V. Balaji, I. Biswas, D.S. Nagaraj,
	\emph{Principal bundles over projective manifolds with parabolic
	structure over a divisor}, Tohoku Math. Jour. 53 (2001), 337–367.

     \bibitem{BBN03}
    V. Balaji, I. Biswas, D.S. Nagaraj, \emph{Ramified G–bundles as parabolic bundles}, Jour. Ramanujan Math. Soc. 18 (2003), 123–138.
    
    \bibitem{BBB01}
     V. Balaji, S. del Baño, I. Biswas, 
     \emph{A Torelli type theorem for the moduli space of parabolic
    vector bundles over curves}, Math. Proc. Cambridge Philos. Soc. 130 (2001), no. 2, 269–280.

    
    \bibitem{Y95}
    K. Yokogawa, \emph{Infinitesimal deformation of parabolic Higgs sheaves}, Internat. J. Math. 6 (1995), 125–148
    
	\bibitem{F93}
	G. Faltings, 
	 \emph{Stable G-bundles and projective connections}, 
	 Jour. Algebraic Geom. 2 (1993), 507–568.

	\bibitem{H87}
	N. J. Hitchin, 
	\emph{Stable bundles and integrable systems}, Duke Math. Jour. 54 (1987), 91–114.
	
	\bibitem{S95}
	C. Simpson, 
	\emph{Moduli of representations of the fundamental group of a smooth projective variety II}, Publi. Math I.H.E.S 80 (1995), 5–79.
	
	\bibitem{G01}
	V. Ginzburg, 
	\emph{The global nilpotent variety is Lagrangian}, Duke Mathematical
	Journal, 109(3) (2001), 511–519.
	
	\bibitem{M94}
	E. Markman,
	\emph{Spectral curves and integrable systems},
	Compositio Mathematica, tome 93, no 3 (1994), 255-290.
	
	\bibitem{R20}
	S. Roy,
	\emph{Hitchin fibration on moduli of symplectic and orthogonal parabolic Higgs bundles}, Math Phys Anal Geom 23, 41 (2020). https://doi.org/10.1007/s11040-020-09366-y
	

\end{thebibliography}
\end{document}